\documentclass[leqno, 12pt]{article}
\usepackage{amsthm}
\usepackage{amssymb}
\usepackage{amsmath}
\usepackage[english]{babel}
\usepackage[utf8]{inputenc}
\usepackage[T1]{fontenc}
\usepackage[normalem]{ulem}
\usepackage{graphicx}
\usepackage{fullpage}
\usepackage{enumitem}
\usepackage{calc}
\usepackage{caption}
\usepackage{subcaption}
\usepackage{hyperref}
\usepackage{times}
\usepackage{float}
\usepackage{marvosym}
\usepackage{hyperref}
\usepackage{xhfill}
\usepackage{tikz}
\usetikzlibrary{shapes,fit,positioning,calc,matrix,arrows}
\tikzset{
  optree/.style={scale=.5,thick,grow'=up,level distance=10mm,inner sep=1pt},
  comp/.style={draw=none,circle,fill,line width=0,inner sep=0pt},
  dot/.style={draw,circle,fill,inner sep=0pt,minimum width=3pt},
  circ/.style={draw,circle,inner sep=1pt,minimum width=4mm},
  emptycirc/.style={draw,circle,inner sep=1pt,minimum width=2mm},
  root/.style={level distance=10mm,inner sep=1pt},
  leaf/.style={draw=none,circle,fill,line width=0,inner sep=0pt},
  nodot/.style={draw,circle,inner sep=1pt},
}

\newtheorem*{cnj*}{Conjecture}
\newtheorem*{thm*}{Theorem}
\newtheorem{thm}{Theorem}[section]
\newtheorem{crl}[thm]{Corollary}
\newtheorem{prp}[thm]{Proposition} 
\newtheorem{lem}[thm]{Lemma}
\newtheorem{cnj}[thm]{Conjecture}

\theoremstyle{definition}
\newtheorem{dfn}[thm]{Definition}

\theoremstyle{remark}
\newtheorem{rmk}[thm]{Remark}

\DeclareMathOperator{\rev}{rev}
      
\newcommand{\operadia}{\href{https://operadia.pythonanywhere.com}{Operadia}}
\newcommand{\oeis}[1]{\href{https://oeis.org/#1}{{#1}}}

\newcommand{\Pp}{\mathcal{P}}
\newcommand{\F}{\mathcal{F}}
\newcommand{\R}{\mathcal{R}}

\newcommand{\LL}{\mathrm{LL}}
\newcommand{\RR}{\mathrm{RR}}
\newcommand{\RL}{\mathrm{RL}}

\newcommand{\NAP}{\mathrm{NAP}}
\newcommand{\Ass}{\mathrm{Ass}}
\newcommand{\Pois}{\mathrm{Poiss}}
\newcommand{\PL}{\mathrm{PreLie}}
\newcommand{\Lie}{\mathrm{Lie}}
\newcommand{\Com}{\mathrm{Com}}
\newcommand{\Perm}{\mathrm{Perm}}
\newcommand{\LA}{\mathrm{LieAdm}}
\newcommand{\Mag}{\mathrm{Mag}}
\newcommand{\CMag}{\mathrm{ComMag}}
\newcommand{\AMag}{\mathrm{SkewMag}}
\newcommand{\Nil}{\mathrm{Nil_2}}
\newcommand{\ANil}{\mathrm{SkewNil_2}}
\newcommand{\CNil}{\mathrm{ComNil_2}}
\newcommand{\Leib}{\mathrm{Leib}}
\newcommand{\Zinb}{\mathrm{Zinb}}
\newcommand{\LeftNil}{\mathrm{LeftNil}}
\newcommand{\KDNAP}{\mathrm{NAP}^!}
\newcommand{\Alia}{\mathrm{Alia}}
\newcommand{\LeftAlia}{\mathrm{LeftAlia}}
\newcommand{\KDAlia}{\mathrm{Alia}^!}
\newcommand{\LieAdm}{\mathrm{LieAdm}}
\newcommand{\KDLieAdm}{\mathrm{LieAdm}^!}
\newcommand{\Bess}{\mathrm{Bess}}

\newcommand{\ditto}[1][.4pt]{\xrfill{#1}~\textquotedbl~\xrfill{#1}}

\title{On Hilbert series of Koszul operads and a classification result for set-operads}
\date{}
\author{Paul Laubie}

\begin{document}

\maketitle

\abstract{Motivated by numerous examples in the literature, we state a conjecture on the Hilbert series of Koszul symmetric operads generated by one element of arity $2$. 
We prove this conjecture for all Koszul symmetric set-operads generated by one element of arity $2$ by explicitly classifying those. 
There are $11$ such operads; $4$ of them are new.}

\section{Introduction}
The Koszul duality and the Koszul property were originally introduced by Priddy in 1970 \cite{P70} for algebras in the non-homogenous case, namely the quadratic-linear case.
They were then extended to the homogenous case for quadratic operads generated by elements of arity $2$ by Ginzburg and Kapranov in 1994 \cite{GK94} and for general quadratic operads by Getzler and Jones in 1994 \cite{GJ94}. 
Koszul duality and Koszul property are very useful tools in homological algebra; indeed, they allow, for  example, the computation of manageable resolutions for algebras and operads. In \cite[Conjecture 7.1]{QA}, Polishchuk and Positselski proposed the following remarkable conjecture:
\begin{cnj*}
  Let $A$ be a finitely generated Koszul algebra, then the Hilbert series of $A$ is a rational function.
\end{cnj*} 
This conjecture is still open and can be extended to symmetric operads following the insight of Khoroshkin and Piontkovski \cite{KP}:
\begin{cnj*}
  Let $\Pp$ a finitely generated Koszul symmetric operad, then the Hilbert series of $\Pp$ is differential algebraic over $\mathbb{Z}[t]$. 
\end{cnj*}
However, in the case of Koszul operads generated by one element of arity $2$, one may expect a stronger result to hold. The number of examples of Koszul symmetric operads generated by one element of arity $2$ appearing in the literature grows, and in most cases the Hilbert series of these operads can be computed; see \cite{Operadia,Zin}. 
In all the examples we know, the Hilbert series of these operads are differential algebraic of order $1$ over $\mathbb{Q}[t]$. We conjecture that this is always the case:
\begin{cnj*}\ref{cnj:main}
  Let $\Pp$ a Koszul symmetric operad generated by one element in arity $2$, then the Hilbert series of $\Pp$ is differential algebraic of order $1$ over $\mathbb{Z}[t]$. 
\end{cnj*}
A classification of Koszul symmetric operads generated by one element of arity $2$ would be enough to prove or disprove this conjecture. 
Some classification results on Koszul operads under some hypothesis already exist; indeed, Bremner and Dotsenko \cite{Assdep} classified all the Koszul operads of associator-dependent algebras generated by one element of arity $2$, which all satisfy Conjecture \ref{cnj:main}. 
In this paper, we classify all the Koszul symmetric set-operads generated by one element of arity $2$ and check our conjecture on those operads. We get the following classification:
\begin{thm*}\ref{thm:main}
    Let $\Pp$ a Koszul set-operad generated by one element of arity $2$, then $\Pp$ is either isomorphic to one of the $7$ operads $\Mag$, $\NAP$, $\Ass$, $\Perm$, $\LA^!$, $\CMag$ or $\Com$, or to one of the $4$ following operads:
    \begin{itemize}
      \item $\Pp_{10}$ which is built from $\CMag$ and $\ANil$ with the relation $[a.b, c]=0$;
      \item $\Pp_{2;2}$ which is built from $\CMag$ and $\ANil$ with the relation $[a, b].c=0$;
      \item $\Pp_{11}$ which is the connected sum of $\CMag$ and $\AMag$;
      \item $\Pp_{2;10}$ which is the connected sum of $\CMag$ and $\ANil$.
    \end{itemize}
\end{thm*}
Four new Koszul symmetric set-operads generated by one element of arity $2$ are discovered, all satisfying our conjecture.

\section{Koszul property}\label{sec:Kp}

The Koszul property is a well studied property on operads; see \cite{AlgOp} for an introduction to operads and Koszul duality for operads. Most of the usual operads like $\Ass$, $\Lie$, or $\Com$ are known to be Koszul. It is a well known fact that Koszul operads are quadratic. However, there is no general method to determine if an operad is Koszul or not.

The main tool to prove that a quadratic operad is not Koszul is the Ginzburg-Kapranov criterion using Hilbert series of the operad.
\begin{prp}
	Let $\Pp$ a Koszul operad generated by elements of arity $2$, and $\Pp^!$ its Koszul dual, then:
	$$f_{\Pp^!}(-f_\Pp(-t))=t.$$
	In particular, if the compositional reverse of $-f_{\Pp}(-t)$ has a nonpositive term, then $\Pp$ is not Koszul.
\end{prp}

Thus any operad generated in arity $2$ such that their Hilbert series is one of the series given in Table~\ref{fg:Hseries} of the \hyperlink{Appendix}{Appendix} is not Koszul.

The main tool we will use to prove that an operad is Koszul is the following property:

\begin{dfn}
	Let $\Pp$ be an operad; $\Pp$ is \emph{monomial} if it admits a presentation: 
	$$\Pp\simeq\F(x^1, \dots, x^n)/\langle \R \rangle,$$ 
	such that $\R$ is a set of monomials. $\Pp$ is \emph{quadratic monomial} if $\R$ is a set of quadratic monomials.
\end{dfn}

\begin{prp}\label{Proposition:K}
  Let $\Pp$ be a quadratic monomial operad, then $\Pp$ is Koszul.
\end{prp}

This fact is a direct consequence of \cite[Theorem 6.3.3.2]{AlgComp} or \cite[Theorem 8.3.1]{AlgOp}. 
However, if one inspects the proof of those theorems, the above proposition is, in fact, one of the first intermediate results of those proofs.

This proposition already allows us to prove the Koszulness of some natural to consider operads $\Mag$, $\CMag$ and $\AMag$ the free operads over respectively one generator of arity $2$ without symmetries, one commutative generator of arity $2$ and one antisymmetric generator of arity $2$, and their Koszul dual $\Nil$, $\ANil$ and $\CNil$.

\section{Koszul symmetric operads generated by one element of arity $2$}

The conjecture expected in \cite{KP} is the following:
\begin{cnj}
  Let $\Pp$ a finitely generated Koszul symmetric operad, then the Hilbert series of $\Pp$ is differential algebraic over $\mathbb{Z}[t]$. 
\end{cnj}
It states that the Hilbert series of any finitely generated Koszul symmetric operad should satisfy a non-trivial differential algebraic equation over $\mathbb{Z}[t]$. 
It can be experimentally checked on examples appearing in the literature; however, most of the examples are operads generated by elements of arity $2$ and are often generated by one element. 
We have already introduced the Koszul operads $\Mag$, $\CMag$, $\AMag$, $\Nil$, $\ANil$ and $\CNil$ which are all generated by one element of arity $2$. 
In order to be as exhaustive as possible, we looked through \operadia  \cite{Operadia} which is an under-construction database of operads and their properties.
One can find $17$ other Koszul operads generated by one element of arity $2$ in \cite{Operadia}; a summary table of those operads and their Hilbert series is given in Table \ref{fg:table} of the \hyperlink{Appendix}{Appendix}, and the sequence of arity-wise dimensions of the operads are given in the OEIS; they might be shifted by $1$ or have a non-zero first term.

One may check that all their Hilbert series satisfy differential algebraic identities. 
Moreover, those identities are of \emph{order $1$}, meaning that only $f$ and $f'$ appear and no higher differential of $f$. 
It leads us to the following conjecture:
\begin{cnj}\label{cnj:main}
  Let $\Pp$ a Koszul symmetric operad generated by one element of arity $2$, then the Hilbert series of $\Pp$ is differential algebraic of order $1$ over $\mathbb{Z}[t]$. Equivalently, $f_\Pp$ and $f_\Pp'$ are algebraically dependent over $\mathbb{Z}[t]$.
\end{cnj}

\begin{rmk}
  One cannot expect this conjecture to be true for Koszul symmetric operads generated by finitely many elements of arity $2$ since one can define the operad $\Com\circ\Com$ by generators and relations:
  $$ \Com\circ\Com=\F(c_1,c_2)/\langle c_1\circ_1 c_1 - c_1\circ_2 c_1,\; c_2\circ_1 c_2 - c_2\circ_2 c_2,\; c_1\circ c_2\rangle $$
  where the action of $\mathfrak{S}_2$ on $\{c_1,c_2\}$ is given by $c_i.(1\;2)=c_i$. This operad is Koszul by the theory of distributed law \cite{distr}; moreover, its Hilbert series is $\exp(\exp(t)-1)-1$, which is not algebraically dependent with its differential over $\mathbb{Z}[t]$.\\
  More generally this construction allow us to build a Koszul operad $\Pp$ with $n$ generators such that its Hilbert series is not algebraically dependent with its first $n-1$ differentials over $\mathbb{Z}[t]$.
\end{rmk}

\section{The magmatic Operad}

As stated in the abstract, a goal of this paper is to classify set-operads generated by one element of arity $2$ that are Koszul. The case with a symmetric generator will not be developed here since the only quadratic set-operads generated by one symmetric element of arity $2$ are $\Com$ and $\CMag$, and are both known to be Koszul.

The following lemma allows us to reduce the study of Koszul set-operads generated by one element of arity $2$ to the study of a finite number of operads.

\begin{lem}\label{Lemma:KS}
  Let $\Pp$ be a Koszul set-operad generated by one element of arity $2$, then $\Pp$ is a quotient of $\Mag$ by an equivariant equivalence relation on the monomials of $\Mag(3)$.
\end{lem}

\begin{proof}
  Since $\Pp$ is generated by one element of arity $2$, it is a quotient of $\Mag$, moreover since $\Pp$ is Koszul, $\Pp$ is quadratic, hence the relations are quadratic. Moreover, since $\Pp$ is set-theoretical, the relations are given by an equivariant equivalence relation on the monomials.
\end{proof}

Some Koszul set-operads generated by one element of arity $2$ are well known:
\begin{itemize}
  \item The magmatic operad on one generator without symmetries, 
  \item The non-associative permutative operad, 
  \item The associative operad, 
  \item The permutative operad, 
  \item The Koszul dual of the Lie admissible operad (which is indeed a set-operad when 
  considering the non-symmetric generator).
\end{itemize}

We will add four new operads to this list. 
Namely, the connected sum of $\CMag$ and $\AMag$, the connected sum of $\CMag$ and $\ANil$, and two other operads built over $\CMag$ and $\ANil$. 
One may notice that neither $\AMag$, $\CNil$, nor $\ANil$ are set-operads; however, they are the building blocks of the $4$ new Koszul set-operads we will construct.

Let $\Mag$ be the magmatic operad on one generator without symmetries $x$, and let $y=x.(1\;2)$. 
The operad $\Mag$ admits a non-trivial automorphism given by $x\mapsto y$, let us call it the \emph{reverse automorphism}. 
The operad $\Mag$ is Koszul and is in fact the first example of Koszul set-operad generated by one element of arity $2$. 
By Lemma \ref{Lemma:KS}, we are left to consider quotients of $\Mag$ by an equivariant equivalence relation on the monomials of $\Mag(3)$. 
Let us study the action of $\mathfrak{S}_3$ on the set of monomials of $\Mag(3)$. 
They can be represented by all possible ways to parenthesize the product of three elements, $a$, $b$, and $c$, with $a$, $b$, and $c$ in any order. 
Let us represent it this way:\\

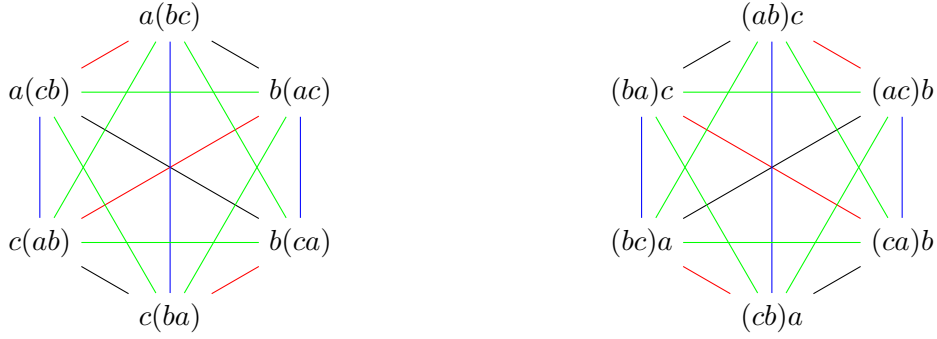
\begin{center}
  %\centering
    \begin{tikzpicture}
      \node (L1) at (-4, 0) {\scalebox{1}{$a(bc)$}};
      \node (L2) at (-2.268, -1) {\scalebox{1}{$b(ac)$}};
      \node (L3) at (-2.268, -3) {\scalebox{1}{$b(ca)$}};
      \node (L6) at (-5.732, -1) {\scalebox{1}{$a(cb)$}};
      \node (L5) at (-5.732, -3) {\scalebox{1}{$c(ab)$}};
      \node (L4) at (-4, -4) {\scalebox{1}{$c(ba)$}};
      \node (R1) at (4, 0) {\scalebox{1}{$(ab)c$}};
      \node (R2) at (5.732, -1) {\scalebox{1}{$(ac)b$}};
      \node (R3) at (5.732, -3) {\scalebox{1}{$(ca)b$}};
      \node (R6) at (2.268, -1) {\scalebox{1}{$(ba)c$}};
      \node (R5) at (2.268, -3) {\scalebox{1}{$(bc)a$}};
      \node (R4) at (4, -4) {\scalebox{1}{$(cb)a$}};
      \draw[red] (R1) -- (R2);
      \draw[red] (R3) -- (R6);
      \draw[red] (R4) -- (R5);
      \draw (R1) -- (R6);
      \draw (R2) -- (R5);
      \draw (R3) -- (R4);
      \draw[blue] (R1) -- (R4);
      \draw[blue] (R2) -- (R3);
      \draw[blue] (R5) -- (R6);
      \draw[green] (R1) -- (R3) -- (R5) -- (R1);
      \draw[green] (R2) -- (R4) -- (R6) -- (R2);
      \draw (L1) -- (L2);
      \draw (L3) -- (L6);
      \draw (L4) -- (L5);
      \draw[red] (L1) -- (L6);
      \draw[red] (L2) -- (L5);
      \draw[red] (L3) -- (L4);
      \draw[blue] (L1) -- (L4);
      \draw[blue] (L2) -- (L3);
      \draw[blue] (L5) -- (L6);
      \draw[green] (L1) -- (L3) -- (L5) -- (L1);
      \draw[green] (L2) -- (L4) -- (L6) -- (L2);
    \end{tikzpicture}
\end{center}
\textcolor{black}{Black} edges represent the action of the transposition $(1\;2)$, \textcolor{red}{red} edges represent the action of the transposition $(2\;3)$, \textcolor{blue}{blue} edges represent the action of the transposition $(1\;3)$, and \textcolor{green}{green} edges represent the action of the $3$-cycles.

As we can see, the action of $\mathfrak{S}_3$ on the monomials of $\Mag(3)$ has $2$ orbits; let us call them the left and right orbits. They are exchanged by the reverse automorphism of $\Mag$.

\begin{prp}\label{Proposition:cases}
  Let $\R$ an equivariant equivalence relation on the monomials of $\Mag(3)$, then $\R$ satisfies exactly one of the following property:
  \begin{enumerate}
    \item all equivalence classes are either subsets of the left orbit or subsets of the right orbit;\label{c:1}
    \item all equivalence classes contain elements of both orbits.\label{c:2}
  \end{enumerate}
\end{prp}

In the first case, the relation is entirely determined by the class of $(ab)c$ and the class of $a(bc)$. In the second case, the class of $(ab)c$ is enough to determine the relation.

The first case can be refined into four subcases:
\begin{prp}\label{Proposition:subcases1}
  Let $\R$ an equivariant equivalence relation on the monomials of $\Mag(3)$ satisfying Property \ref{c:1} of Proposition \ref{Proposition:cases}, then $\R$ satisfies exactly one of the following property:
  \begin{enumerate}[label*=1.\arabic*]
    \item the relation is trivial (equivalence classes are singletons);\label{sc:triv}
    \item equivalence classes of the left orbit are reduced to singletons, and equivalence classes of the right are not;\label{sc:sr}
    \item equivalence classes of the right orbit are reduced to singletons, and equivalence classes of the left are not;\label{sc:sl}
    \item no equivalence class is reduced to a singleton.\label{sc:lr}
  \end{enumerate}
\end{prp}

The first subcase gives rise to the operad $\Mag$. The second and third subcases are equivalent by the reverse automorphism. The last subcase is the same as giving one equivalence relation on the left orbit and another one on the right orbit.

The second case can be refined into two subcases:
\begin{prp}\label{Proposition:subcases2}
  Let $\R$ an equivariant equivalence relation on the monomials of $\Mag(3)$ satisfying Property \ref{c:2} of Proposition \ref{Proposition:cases}, then $\R$ satisfies exactly one of the following property:
  \begin{enumerate}[label*=2.\arabic*]
    \item all equivalence classes contain exactly two elements;\label{sc:simple}
    \item all equivalence classes contain strictly more than two elements (namely $4$, $6$, or $12$ elements). \label{sc:comp}
  \end{enumerate}
\end{prp}

We will study each subcase of Proposition \ref{Proposition:subcases1} and \ref{Proposition:subcases2} in the following section. 
This approach leads us to consider $57$ operads.

\section{Construction of all the quadratic set-operads generated by one element of arity $2$}

Let us first study Subcase \ref{sc:sr}. By equivariance, a relation of Subcase \ref{sc:sr} is entirely determined by the class of $(ab)c$. 
Moreover, the equivalence class of $(ab)c$ contains either $2$, $3$, or $6$ elements. 
We have a priori $5$ possibilities to relate two elements of the right orbit:
\begin{itemize}
  \item $(ab)c \sim (ac)b$;
  \item $(ab)c \sim (ba)c$;
  \item $(ab)c \sim (bc)a$;
  \item $(ab)c \sim (ca)b$;
  \item $(ab)c \sim (cb)a$.
\end{itemize}

However, since the relation is equivariant, the relation generated by $(ab)c \sim (ca)b$ is the same as the relation generated by $(ab)c \sim (bc)a$. Moreover, this relation is the only one such that the equivalence class of $(ab)c$ contains $3$ elements.
We have only one possibility for a relation such that the equivalence class of $(ab)c$ contains $6$ elements:
$$(ab)c \sim (cb)a \sim (ca)b \sim (ba)c \sim (ac)b \sim (bc)a$$
Hence we have $5$ possible relations, which give rise to $5$ operads:
\begin{itemize}
  \item $\RR_1=\{(ab)c = (ac)b\}$, 
  \item $\RR_2=\{(ab)c = (ba)c\}$, 
  \item $\RR_3=\{(ab)c = (cb)a\}$, 
  \item $\RR_4=\{(ab)c = (bc)a = (ca)b\}$, 
  \item $\RR_5=\{(ab)c = (cb)a = (ca)b = (ba)c = (ac)b = (bc)a\}$, 
\end{itemize}
Let $\Pp_{i}=\Mag/\langle \RR_i\rangle$

Let us study Subcase \ref{sc:lr}. We need to study the compatibility between the relations on the left and right orbits. We have $5$ relations involving only the left orbit given by the reverse automorphism applied to the relation of the left orbit:
\begin{itemize}
  \item $\LL_1=\{a(bc) = b(ac)\}$, 
  \item $\LL_2=\{a(bc) = a(cb)\}$, 
  \item $\LL_3=\{a(bc) = c(ba)\}$, 
  \item $\LL_4=\{a(bc) = b(ca) = c(ab)\}$, 
  \item $\LL_5=\{a(bc) = b(ca) = c(ab) = a(cb) = b(ac) = c(ba)\}$, 
\end{itemize}

We have $25$ possibilities to combine the relations $\RR_i$ and $\LL_j$; however, by the reverse automorphism, we only need to study those with $i \leq j$, and thus we have $15$ possibilities. Let $\Pp_{i;j}=\Mag/\langle \RR_i;\LL_j\rangle$.

Let us study equivalence relations involving exactly one term of the left and one of right orbits. As previously, knowing the equivalence class of $(ab)c$ is enough to determine the relation. Thus, we have $6$ possibilities:
\begin{itemize}
  \item $\RL_6 = \{(ab)c = a(bc)\}$, 
  \item $\RL_7 = \{(ab)c = a(cb)\}$, 
  \item $\RL_8 = \{(ab)c = b(ac)\}$, 
  \item $\RL_9 = \{(ab)c = b(ca)\}$, 
  \item $\RL_{10} = \{ (ab)c = c(ab)\}$, 
  \item $\RL_{11} = \{ (ab)c = c(ba)\}$.
\end{itemize}
Moreover, none of those relations are equivalent. We get the $6$ operads $\Pp_6$ up to $\Pp_{11}$ with $\Pp_{i}= \Mag/\langle \RL_i\rangle$.

Let us finally study Subcase \ref{sc:comp}. Because the relations are equivariant, the same number of elements of the left and right orbits must appear in each class of the equivalence relation, either $2$, $3$, or $6$. 
Any relation of Subcase~\ref{sc:comp} is a combination of a relation of Subcase \ref{sc:simple} and a relation of Subcase \ref{sc:sr}. 
Let us define the operads $\Pp_{i;j}= \Mag/\langle \RR_i;\RL_j\rangle$.

\begin{prp}\label{prp:enum}
  Let $\Pp$ be a quadratic set-operad generated by one element of arity $2$, then $\Pp$ is isomorphic to the operad $\Mag$, one of the $11$ operads $\Pp_{i}$, or one of the $45$ operads $\Pp_{i;j}$.
\end{prp}

We get a total of $57$ operads; however, a bunch of them are isomorphic.

\begin{prp}\label{prp:iso}
  The couples of operads $(\Pp_{1;6}, \Pp_{1;7})$, $(\Pp_{1;8}, \Pp_{1;10})$, $(\Pp_{1;9}, \Pp_{1;11})$, $(\Pp_{2;6}, \Pp_{2;8})$, $(\Pp_{2;7}, \Pp_{2;9})$, $(\Pp_{2;10}, \Pp_{2;11})$, $(\Pp_{3;6}, \Pp_{3;11})$, $(\Pp_{3;7}, \Pp_{3;10})$ and $(\Pp_{3;8}, \Pp_{3;9})$ are couples of isomorphic operads. The triples of operads $(\Pp_{4;6}, \Pp_{4;9}, \Pp_{4;10})$ and $(\Pp_{4;7}, \Pp_{4;8}, \Pp_{4;11})$ are triples of isomorphic operads.
\end{prp}

\begin{proof}
  Let us compute the class of $(ab)c$ in $\Pp_{1;6}$. We get $(ab)c=(ac)b$ by definition of $\RR_1$ and $(ab)c=a(bc)$ by definition of $\RL_6$. Moreover, by equivariance, $\RL_6$ implies $(ac)b=a(cb)$, and thus the class of $(ab)c$ is 
	$$\{(ab)c, (ac)b, a(bc), a(cb)\}$$
  Which is the same as in $\Pp_{1;7}$.
  The exact same argument works for the ten other cases.
\end{proof}

We also recognize some known operads:

\begin{lem}\label{lem:known}
  The following operads are isomorphic to Koszul operads listed in \cite{Operadia}:
  \begin{itemize}
    \item The operad $\Pp_1$ is isomorphic to $\NAP$.
    \item The operad $\Pp_6$ is isomorphic to $\Ass$.
    \item The operads $\Pp_{1;6}$ and $\Pp_{2;6}$ are isomorphic to $\Perm$.
    \item The operads $\Pp_{5;6}$, $\Pp_{5;7}$, $\Pp_{5;8}$, $\Pp_{5;9}$, $\Pp_{5;10}$, and $\Pp_{5;11}$ are all isomorphic to $\KDLieAdm$.
  \end{itemize}
\end{lem}

\begin{proof}
	For the first three points, we recover the exact definition of those operads. We refer to Table~\ref{fg:table} and their entry in \operadia \cite{Operadia} for a reference where they are defined, and their Koszulness is proven.

	For the last point, let us polarize the relations. Let $[a, b]=ab-ba$ and $a.b=ab+ba$. Rewriting the relations using those generators gives:
  	\begin{itemize}
    		\item $[a, [b, c]]=0$,
    		\item $[a, b].c=0$,
	    	\item $[a, b.c]=0$,
    		\item $a.(b.c)=(a.b).c$.
  	\end{itemize}
	We recognize the Koszul dual of the Lie admissible operad.
\end{proof}

\section{The non-Koszul operads}

Lemma~\ref{lem:known} and Propositions~\ref{prp:enum} and~\ref{prp:iso} reduce the number of operads to check from $57$ to $33$. We will now study the Koszulness of the remaining operads, starting with the non-Koszul ones. 

\begin{prp}
	The $25$ operads $\Pp_2$, $\Pp_3$, $\Pp_4$, $\Pp_5$, $\Pp_7$, $\Pp_8$, $\Pp_9$, $\Pp_{1;1}$, $\Pp_{1;3}$, $\Pp_{1;2}$, $\Pp_{2;3}$, $\Pp_{1;4}$, $\Pp_{2;4}$, $\Pp_{3;4}$, $\Pp_{4;4}$, $\Pp_{1;5}$, $\Pp_{2;5}$, $\Pp_{3;5}$, $\Pp_{1;8}$, $\Pp_{1;9}$, $\Pp_{2;7}$, $\Pp_{3;6}$, $\Pp_{3;7}$, $\Pp_{3;8}$, and $\Pp_{4,5}$ are not Koszul.
\end{prp}

\begin{proof}
	To show this result, we compute the first few terms of the Hilbert series of each operad either by hand or using the Haskell calculator \cite{Hask}.
	Since all of them appear in Table~\ref{fg:Hseries} of the \hyperlink{Appendix}{Appendix}, those operads are not Koszul. 
\end{proof}

The naïve Hilbert series argument does not seem to work for the $8$ remaining operads, however; either by refining the argument or by computing the Koszul dual, we can show that $4$ of them are not Koszul.

\begin{prp}
  The operad $\Pp_{3;3}$ is not Koszul.
\end{prp}

\begin{proof}
  The method consisting of finding negative coefficients in $\rev(-f_{\Pp}(-t))$ does not seem to work in this case. However, one can notice that this operad is self-dual (in the sense of Koszul duality). We recall that $x$ is the generator without symmetries of this operad; let $y=x.(1\;2)$. The relations of $\Pp_{3;3}$ are:
  $$ x\circ_1 x - y\circ_2 y \qquad ; \qquad y\circ_1 y - x\circ_2 x $$
  And those relations are the same as the one of the operad $\Pp_{3;3}^!$ with $x\mapsto x_*$.\\
  The explicit method to compute a Koszul dual is given in \cite[Subsection 7.6.1]{AlgOp}.\\
  Let us compute the first dimensions of the operad $\Pp_{3;3}$ using the Haskell calculator \cite{Hask}. We get $(1, 2, 6, 20, 60, 182, 546, ?, ?, \dots)$. Thus we can compute:
  $$f_{\Pp_{3;3}}(-f_{\Pp_{3;3}}(-t))=t-\frac{7}{12}t^7+\mathcal{O}(t^8)$$
  Which shows that $\Pp_{3;3}$ is not Koszul.
\end{proof}

\begin{prp}
  The operad $\Pp_{5;5}$ is not Koszul.
\end{prp}

\begin{proof}
  Once again the generating series method does not seem to work in this case. Let us polarize the relations of $\Pp_{5;5}$ by defining $[a, b]=ab-ba$ and $a.b=ab+ba$ and rewriting the relations of $\Pp_{5;5}$ using those:
  \begin{itemize}
    \item $(a.b).c=a.(b.c)$,
    \item $[a.b, c]=-[a, b.c]$,
    \item $[a, b].c=0$,
    \item $[[a, b], c]=0$.
  \end{itemize}
  From this presentation we can see that $\Pp_{5;5}$ is graded by $[\cdot,\cdot ]$ and the dimensions can be easily computed. We get $(1, 1+u, 1+u, 1, 1, 1, \dots)$, where $u$ is the generator of the grading. 
	This case looks very much like \cite[Proposition 3.6]{Assdep} (however, this is not the same operad). We have that:
  $$f_{\Pp_{5;5}}=\exp(t)-1+\frac{u}{2}t^2+\frac{u}{6}t^3$$
  The Ginzburg-Kapranov criterion can be used on this series and the first negative term is $$\frac{-1983044460002323872u^2}{20!}t^{20}$$ which is the exact same term as in \cite[Proposition 3.6]{Assdep} since the power series are the same. Thus $\Pp_{5;5}$ is not Koszul.
\end{proof}

\begin{prp}
  The operad $\Pp_{4;6}$ is not Koszul.
\end{prp}

\begin{proof}
  Let us polarize the relations. We get:
  \begin{itemize}
    \item $[a, [b, c]]=0$,
    \item $[a, b].c=a.[b, c]$,
    \item $[a, b.c]=0$,
    \item $(a.b).c=a.(b.c)$
  \end{itemize}
  We recognize the operad of \cite[Proposition 3.6]{Assdep} which is not Koszul.
\end{proof}
  
\begin{prp}
  The operad $\Pp_{4;9}$ is not Koszul.
\end{prp}
  
\begin{proof}
  Let us polarize the relations. We get:
  \begin{itemize}
    \item $[a, [b, c]]=-[[a, b], c]$, 
    \item $[a, b].c=0$, 
    \item $[a, b.c]=0$, 
    \item $(a.b).c=a.(b.c)$.
  \end{itemize}
	From this presentation, we can notice that $\Pp_{4;9}$ is graded by $[\cdot,\cdot ]$, and the dimensions can be easily computed. 
	We get $(1, 1+u, 1+u^2, 1, 1, 1, \dots)$, where $u$ is the generator of the grading. We have that:
  $$f_{\Pp_{4;9}}=\exp(t)-1+\frac{u}{2}t^2+\frac{u^2}{6}t^3$$ 
  The Ginzburg-Kapranov criterion can be used on this series and the first negative term is $\frac{-35u^5}{6!}t^6$ which shows that this operad is not
  Koszul.
\end{proof}

\section{The Koszul operads}

We have $4$ operads left to check, namely the operads $\Pp_{2,2}$, $\Pp_{10}$, $\Pp_{11}$, and $\Pp_{2,10}$. Let us show that they are Koszul.

\begin{thm}
  The operad $\Pp_{2;2}$ is Koszul and self-dual. 
\end{thm}
    
\begin{proof}
  Let us polarize the relations of $\Pp_{2;2}$. We get:
  \begin{itemize} 
    \item $[a, [b, c]]=0$, 
    \item $[a, b].c=0$, 
  \end{itemize}
  By Proposition \ref{Proposition:K}, $\Pp_{2;2}$ is Koszul. Moreover, its generating series can be computed and is $2-t-\sqrt{1-2t}$. Computation of its Koszul dual shows that it is self-dual. 
\end{proof}

\begin{thm}
  The operad $\Pp_{10}$ is Koszul and is self-dual.
\end{thm}

\begin{proof}
  Let us polarize the relations of $\Pp_{10}$. We get::
  \begin{itemize}
    \item $[a, [b, c]]=0$,
    \item $[a, b.c]=0$,
  \end{itemize}
  By Proposition \ref{Proposition:K}, $\Pp_{10}$ is Koszul. Moreover, its generating series can be computed and is\\ $1-\sqrt{1-2t-t^2}$. Computation of its Koszul dual shows that it is self-dual.
\end{proof}

\begin{thm}
  The operad $\Pp_{11}$ is Koszul.
\end{thm}

\begin{proof}
  Let us polarize the relations of $\Pp_{11}$. We get:
  \begin{itemize}
    \item $a.[b, c]=0$,
    \item $[a, b.c]=0$,
  \end{itemize}
  We get the connected sum of the magmatic operad over a symmetric generator and the magmatic operad over a skew-symmetric generator. By Proposition \ref{Proposition:K}, it is Koszul. Moreover, its generating series can be computed and is $2-t-\sqrt{1-2t}$.
\end{proof}

\begin{thm}
  The operad $\Pp_{2;10}$ is Koszul.
\end{thm}

\begin{proof}
  Let us polarize the relations. We get:
  \begin{itemize}
    \item $[[a, b], c]=0$,
    \item $[a, b].c=0$,
    \item $[a, b.c]=0$.
  \end{itemize}
  One may recognize the connected sum of $\CMag$ and $\ANil$. By Proposition \ref{Proposition:K}, $\Pp_{2;10}$ is Koszul. Moreover, its generating series can be computed and is $1-\sqrt{1-2t}+\frac{1}{2}t^2$.
\end{proof}

All the cases have been checked. To compute the formal power series of those operads, we use two construction. The connected sum, denoted $\#$, which is the coproduct of unital algebraic operads, and distributive laws. We refer to \cite[Section 8.6]{AlgOp} for an introduction to distributive laws.
We can now state the main theorem of this paper:

\begin{thm}\label{thm:main}
  Let $\Pp$ a Koszul set-operad generated by one element of arity $2$, then $\Pp$ is isomorphic to one of the $11$ following operads:
  \begin{itemize}
    \item $\Mag$ the magmatic operad, and $f_\Pp(t)=\frac{1}{2}(1-\sqrt{1-4t})$;
    \item $\NAP$ the non-associative permutative operad, and $f_\Pp$ is the Euler tree function;
    \item $\Pp_{10}=\CMag\circ\ANil$ with the relation $[a.b, c]=0$, and $f_\Pp(t)=1-\sqrt{1-2t-t^2}$;
    \item $\Pp_{2;2}=\ANil\circ\CMag$ with the relation $[a, b].c=0$, and $f_\Pp(t)=2-t-\sqrt{1-2t}$;
    \item $\Pp_{11}=\CMag\#\AMag$, and $f_\Pp(t)=2-t-\sqrt{1-2t}$;
    \item $\Ass$ the associative operad, and $f_\Pp(t)=\frac{t}{1-t}$;
    \item $\Pp_{2;10}=\CMag\#\ANil$, and $f_\Pp(t)=1-\sqrt{1-2t}+\frac{1}{2}t^2$;
    \item $\Perm$ the permutative operad and $f_\Pp(t)=t\exp(t)$;
    \item $\LA^!$ the Koszul dual of the Lie admissible operad, and $f_\Pp(t)=\exp(t)-1+\frac{t^2}{2}$;
    \item $\CMag$ the commutative magmatic operad, and $f_\Pp(t)=(1-\sqrt{1-2t})$;
    \item $\Com$ the commutative operad, and $f_\Pp(t)=\exp(t)-1$.
  \end{itemize}
  Moreover, only $\Ass$, $\Pp_{2;2}$ and $\Pp_{10}$ are self-dual. And only $\NAP$ and $\Perm$ do not inherit the reverse automorphism from $\Mag$.
\end{thm}

All those operads satisfy Conjecture \ref{cnj:main}.

\begin{crl}
  The Hilbert series of a Koszul symmetric set-operad generated by one element of arity $2$ is differential algebraic of order $1$ over $\mathbb{Z}[t]$.
\end{crl}

\subsection*{Funding}
{\small The author was partially supported by the ANR project HighAGT (ANR-20-CE40-0016) and is funded by a postdoctoral fellowship of the ERC Starting Grant “Low Regularity Dynamics via Decorated Trees” (LoRDeT) of Yvain Bruned (grant agreement No. 101075208).}

\bibliographystyle{plain}
\bibliography{Bibly}
\textsc{Institut Élie Cartan de Lorraine, UMR 7502, Faculté des Sciences et Technologies, Boulevard des Aiguillettes, 54506 Vandœuvre-lès-Nancy, France}
\textit{Email address:} \texttt{paul.laubie@univ-lorraine.fr}

\section*{\hypertarget{Appendix}{Appendix}}

\begin{table}\caption{Summary table of Hilbert series failing the Ginzburg-Kapranov criterion and their first obstruction.}\label{fg:Hseries}
  {\renewcommand{\arraystretch}{2}
  \begin{tabular}{l|l|l}
	  $f_\Pp(t)$ & first negative term & Operad\\ [-10pt]
	  & in $\rev_t(-f_\Pp(-t))$ &  \\
    \hline
	  $t+\frac{2}{2!}t^2+\frac{9}{3!}t^3+\frac{60}{4!}t^4+\frac{525}{5!}t^5+\mathcal{O}(t^6)$ & $\frac{-15}{5!}t^5$ & $\Pp_2$ \\
	  $t+\frac{2}{2!}t^2+\frac{9}{3!}t^3+\frac{60}{4!}t^4+\frac{520}{5!}t^5+\mathcal{O}(t^6)$ & $\frac{-10}{5!}t^5$ & $\Pp_3$ \\
	  $t+\frac{2}{2!}t^2+\frac{8}{3!}t^3+\frac{40}{4!}t^4+\frac{210}{5!}t^5+\mathcal{O}(t^6)$ & $\frac{-50}{5!}t^5$ & $\Pp_4$ \\
	  $t+\frac{2}{2!}t^2+\frac{7}{3!}t^3+\frac{29}{4!}t^4+\frac{146}{5!}t^5+\mathcal{O}(t^6)$ & $\frac{-46}{5!}t^5$ & $\Pp_5$ \\
	  $t+\frac{2}{2!}t^2+\frac{6}{3!}t^3+\frac{12}{4!}t^4+\frac{20}{5!}t^5+\mathcal{O}(t^6)$ & $\frac{-140}{5!}t^5$ & $\Pp_7,\Pp_8$ \\
	  $t+\frac{2}{2!}t^2+\frac{6}{3!}t^3+\frac{12}{4!}t^4+\frac{1}{5!}t^5+\mathcal{O}(t^6)$ & $\frac{-121}{5!}t^5$ & $\Pp_9$ \\
	  $t+\frac{2}{2!}t^2+\frac{6}{3!}t^3+\frac{14}{4!}t^4+\frac{30}{5!}t^5+\mathcal{O}(t^6)$ & $\frac{-90}{5!}t^5$ & $\Pp_{1;1},\Pp_{1;3}$ \\
	  $t+\frac{2}{2!}t^2+\frac{6}{3!}t^3+\frac{20}{4!}t^4+\frac{75}{5!}t^5+\frac{312}{6!}t^6+\mathcal{O}(t^7)$ & $\frac{-318}{6!}t^6$ & $\Pp_{1;2}$ \\
	  $t+\frac{2}{2!}t^2+\frac{6}{3!}t^3+\frac{14}{4!}t^4+\frac{21}{5!}t^5+\mathcal{O}(t^6)$ & $\frac{-81}{5!}t^5$ & $\Pp_{2;3}$ \\
	  $t+\frac{2}{2!}t^2+\frac{5}{3!}t^3+\frac{6}{4!}t^4+\frac{10}{5!}t^5+\frac{18}{6!}t^6+\mathcal{O}(t^7)$ & $\frac{-2572}{6!}t^6$ & $\Pp_{1;4}$ \\
	  $t+\frac{2}{2!}t^2+\frac{5}{3!}t^3+\frac{8}{4!}t^4+\frac{18}{5!}t^5+\frac{55}{6!}t^6+\mathcal{O}(t^7)$ & $\frac{-1541}{6!}t^6$ & $\Pp_{2;4}$ \\
	  $t+\frac{2}{2!}t^2+\frac{5}{3!}t^3+\frac{2}{4!}t^4+\frac{2}{5!}t^5+\mathcal{O}(t^6)$ & $\frac{-112}{5!}t^5$ & $\Pp_{3;4}$ \\
	  $t+\frac{2}{2!}t^2+\frac{4}{3!}t^3+\frac{2}{4!}t^4+\frac{2}{5!}t^5+\frac{2}{6!}t^6+\frac{2}{7!}t^7+\mathcal{O}(t^8)$ & $\frac{-26238}{7!}t^7$ & $\Pp_{4;4}$ \\
	  $t+\frac{2}{2!}t^2+\frac{4}{3!}t^3+\frac{5}{4!}t^4+\frac{6}{5!}t^5+\frac{7}{6!}t^6+\frac{8}{7!}t^7+\frac{9}{8!}t^8+\mathcal{O}(t^9)$ & $\frac{-95669}{8!}t^8$ & $\Pp_{1;5},\Pp_{2;5}$ \\
	  $t+\frac{2}{2!}t^2+\frac{4}{3!}t^3+\frac{2}{4!}t^4+\frac{1}{5!}t^5+\frac{1}{6!}t^6+\frac{1}{7!}t^7+\mathcal{O}(t^8)$ & $\frac{-29093}{7!}t^7$ & $\Pp_{3;5}$ \\
	  $t+\frac{2}{2!}t^2+\frac{3}{3!}t^3+\frac{1}{4!}t^4+\frac{1}{5!}t^5+\frac{1}{6!}t^6+\frac{1}{7!}t^7+\frac{1}{8!}t^8+\frac{1}{9!}t^9$ & $\frac{-802543633}{11!}t^{11}$ & $\Pp_{1;8},\Pp_{1;9},\Pp_{2;7}$ \\[-10pt]
	  $+\frac{1}{10!}t^{10}+\frac{1}{11!}t^{11}+\mathcal{O}(t^{12})$ & & $\Pp_{3;6},\Pp_{3;7},\Pp_{3;8}$ \\
	  $t+\frac{2}{2!}t^2+\frac{3}{3!}t^3+\frac{2}{4!}t^4+\frac{2}{5!}t^5+\frac{2}{6!}t^6+\frac{2}{7!}t^7+\frac{2}{8!}t^8+\frac{2}{9!}t^9$ & $\frac{-1080639958361062}{15!}t^{15}$ & $\Pp_{4;5}$ \\[-10pt]
    $+\frac{2}{10!}t^{10}+\frac{2}{11!}t^{11}+\frac{2}{12!}t^{12}+\frac{2}{13!}t^{13}+\frac{2}{14!}t^{14}+\frac{2}{15!}t^{15}+\mathcal{O}(t^{16})$ & \\
  \end{tabular}}
\end{table}

\begin{table}\caption{Summary table of Koszul operads generated by one element of arity $2$ in \operadia.}\label{fg:table}
  \renewcommand{\arraystretch}{2}
  \scalebox{.84}{
  \begin{tabular}{l|l|l|l|l}
      Operad & Entry in the OEIS & Hilbert series & Equation & Reference\\ 
      \hline
      $\Lie$      & \oeis{A000142}  & $-\ln(1-t)$                                   & $(1-t)f'-1$                             &  \\
      $\Com$      & \oeis{A000012}  & $\exp(t)-1$                                   & $f-f'+1$                                &  \\
      $\Ass$      & \oeis{A000142}  & $\frac{t}{1-t}$                               & $(1-t)f-t$                              & \cite{Ass} \\
      $\Pois$     & \ditto          & \ditto                                        & \ditto                                  & \cite{Poiss} \\
      $\Leib$     & \ditto          & \ditto                                        & \ditto                                  & \cite{Leib} \\
      $\Zinb$     & \ditto          & \ditto                                        & \ditto                                  & \cite{Zinb} \\
      $\LeftNil$  & \ditto          & \ditto                                        & \ditto                                  & \cite{LeftNil} \\
      $\PL$       & \oeis{A000169}  & $\rev_t(t\exp(t))$                              & $tff'-tf'+f$                            & \cite{PreLie} \\
      $\NAP$      & \ditto          & \ditto                                        & \ditto                                  & \cite{NAP} \\
      $\Perm$     & \oeis{A000027}  & $t\exp(t)$                                    & $(1+t)f-tf' $                           & \cite{Perm} \\
      $\KDNAP$    & \ditto          & \ditto                                        & \ditto                                  & \cite{NAP} \\
      % $\Mag$      & \oeis{A001813}  & $\frac{1-\sqrt{1-4t}}{2}$                     & $f^2-f+t$                               & \\
      % $\Nil$      & None            & $t+t^2$                                       & $f-(t+t^2)$                             & \\
      % $\CMag$     & \oeis{A001147}  & $(1-\sqrt{1-2t})$                             & $f-\frac{f^2}{2}-t$                     & \\
      % $\AMag$     & \ditto          & \ditto                                        & \ditto                                  & \\
      % $\ANil$     & None            & $t+\frac{t^2}{2}$                             & $f-\left(t+\frac{t^2}{2}\right)$        & \\
      % $\CNil$     & \ditto          & \ditto                                        & \ditto                                  & \\
      $\Alia$     & \oeis{A220433}  & $\rev_t\left(-t+t^2-\frac{t^3}{6}\right)$       & $f-f^2+\frac{f^3}{6}-t$                 & \cite{Alia} \\
      $\LeftAlia$ & \ditto          & \ditto                                        & \ditto                                  & \cite{Alia} \\
      $\KDAlia$   & None            & $t+t^2+\frac{t^3}{6}$                         & $f-\left(t+t^2+\frac{t^3}{6}\right)$    & \\
      $\LieAdm$   & \oeis{A337017}  & $\rev_t\left(1-\frac{t^2}{2} - \exp(-t)\right)$ & $f'\left(\frac{f^2}{2}+f+t-1\right)+1$  & \cite{LieAdm} \\
      $\KDLieAdm$ & \oeis{A294619}  & $\frac{t^2}{2} + \exp(t)-1$                   & $f-f'-\left(\frac{t^2}{2}-2t-1\right) $ & \\
      $\Bess$     & \oeis{A001515}  & $\exp(1-\sqrt{1-2t})-1$                       & $(1-2t)(f'-f-1)^2-(f+1)^2 $             & \cite{Bess}
  \end{tabular}}
\end{table}

\end{document}